\documentclass[11pt]{article}%
\usepackage[a4paper,left=1.8cm,right=1.8cm,top=2cm]{geometry}
\usepackage[english]{babel}
\usepackage{amsfonts, amsmath, amssymb,amscd}
\usepackage[all]{xy}
\usepackage{latexsym}
\usepackage{amsmath, amsthm}
\usepackage{amsfonts}
\usepackage{amssymb}
\usepackage{amsmath}
\usepackage{graphicx,color}
\newtheorem{theorem}{Theorem}[section]
\newtheorem{lemma}[theorem]{Lemma}
\newtheorem{claim}{Claim}
\newtheorem{proposition}[theorem]{Proposition}
\newtheorem{definition}[theorem]{Definition}

\newtheorem{remark}[theorem]{Remark}

\newcommand{\R}{\mathbb{R}}
\newcommand{\Hi}{\mathbb{H}}

\newcommand{\arcsinh}{{\rm arcsinh}}

\definecolor{mycolor1}{rgb}{0.3,0.3,1}

\definecolor{mycolor2}{rgb}{0.7,0.3,0}

\usepackage{verbatim}

\begin{document}


\title{Bounded $\lambda$-harmonic functions in domains of $\mathbb{H}^n$ with asymptotic boundary with fractional dimension}
\author{Leonardo Prange Bonorino
\and Patr\'{\i}cia Kruse Klaser  }

\maketitle

\begin{abstract}
The existence and nonexistence of $\lambda$-harmonic functions in unbounded domains of $\mathbb{H}^n$ are investigated. We prove that if the $(n-1)/2$ Hausdorff measure  of the asymptotic boundary of a domain $\Omega$ is zero, then there is no bounded $\lambda$-harmonic function of $\Omega$ for $\lambda \in [0,\lambda_1(\mathbb{H}^n)]$, where $\lambda_1(\mathbb{H}^n)=(n-1)^2/4$.  
For these domains, we have comparison principle and some maximum principle. Conversely, for any $s>(n-1)/2,$ we prove the existence of domains with asymptotic boundary of dimension $s$ for which there are bounded $\lambda_1$-harmonic functions that decay exponentially at infinity. \end{abstract}

\section{Introduction}\label{sec-intro}

Let $\mathbb{H}^n$ be the hyperbolic space and let $\Omega$ be a domain (open connected set) in $\mathbb{H}^n$. 
For $\Omega \ne \mathbb{H}^n$, we say that a nontrivial $u$ is a $\lambda$-harmonic function of $\Omega$ if $u\in C(\overline{\Omega}) \cap C^2(\Omega)$ satisfies
\begin{equation}
 \left\{ \begin{array}{rcl}  -\Delta u & = & \lambda u \quad {\rm in } \quad \Omega \\[5pt]
                                  u & = & 0 \quad {\rm on } \quad \partial \Omega. \end{array} \right.
\label{definitionOfLambdaHarmonic}
\end{equation}
In the case $\Omega=\mathbb{H}^n$ we only require that $u\not\equiv 0$ satisfies the equation above. 
Remind that this is a classical eigenvalue problem when $\Omega$ is a bounded domain. In this case $u$ is an eigenfunction associated to the eigenvalue $\lambda$
and, from the Spectral Theory, the set of eigenvalues is discrete and it is the spectrum of $-\Delta:H_0^1(\Omega) \to H_0^1(\Omega),$ where $H_0^1(\Omega)$
is the Sobolev space that is the closure of $C_0^\infty(\Omega)$ with the $L^2$ norm of the gradient.
However the situation is different for unbounded domains and this characterization of the spectrum does not hold. Indeed, there exists $\lambda$ for which problem \eqref{definitionOfLambdaHarmonic} has nontrivial solution but still $\lambda$ is neither an eigenvalue nor an element of the essential spectrum.

Even then, several results 
are known. For instance, if $\Omega$ is the whole $\mathbb{H}^n$, a class of $\lambda$-harmonic functions is obtained in \cite{H}, \cite{M} using some Poisson integral representation. 
This representation is used in \cite{GO} to characterize the bounded $\lambda$-harmonic functions for any $\lambda \in \mathbb{C}$.

For a Hadamard manifold $M$ with sectional curvature bounded from above and below by negative constants, Ancona proved in \cite{AA} the existence of bounded $\lambda$-harmonic functions in $M$, that converge to zero at infinity with exponential rate, for any $\lambda \in (0,\lambda_1)$, where $\lambda_1$ is the first eigenvalue of $M$ defined by
$$ \lambda_1(M) = \inf_{v \in H^1(M)}  \frac{\displaystyle \int_M |\nabla v |^2 dx}{\displaystyle \int_M |v|^2 dx.} $$
On the other hand, he also exhibited examples of manifolds for which there is no positive $\lambda_1$-harmonic function 
 that converges to $0$ at the asymptotic boundary. This illustrates the importance of the first eigenvalue in the behavior of the $\lambda$-harmonic functions in general manifolds. The asymptotic boundary with the cone topology \cite{EO} also plays an important role in the estimates needed to the main results of \cite{AA1} and \cite{AA}.

The purpose of this work is to study the existence of bounded $\lambda$-harmonic functions in unbounded domains
of $\mathbb{H}^n$ for $\lambda \le \lambda_1(\mathbb{H}^n)$, where the first eigenvalue of $\mathbb{H}^n$ is evaluated by McKean in \cite{Mk}:
$$ \lambda_1(\mathbb{H}^n)=\frac{(n-1)^2}{4}.$$
We show that if the asymptotic boundary boundary is not ``large enough'', the problem has no bounded solution.
This extends our knowledge that for bounded domains (or domains with empty asymptotic boundary) there are no $\lambda$-harmonic functions, if $\lambda \le \lambda_1.$ One of our main results is the following theorem, that is not restricted to functions that converge to zero at the asymptotic boundary:

\begin{theorem}
Let $\Omega \subset \mathbb{H}^n$ such that the $(n-1)/2$ Hausdorff dimensional measure of $\partial_{\infty} \Omega,$ $H^{(n-1)/2}(\partial_{\infty} \Omega)$ is zero. Then there is no bounded $\lambda$-harmonic function that
vanishes on $\partial \Omega$ for $\lambda \in [0,\lambda_1]$. 
\label{Theo_NoBoundedSolutionForNullHausdorff}
\end{theorem}

A consequence of Theorem \ref{Theo_NoBoundedSolutionForNullHausdorff} is that for domains $\Omega$ with $H^{(n-1)/2}(\partial_{\infty} \Omega)=0$ we have a comparison principle, that is, if $u_1$ and $u_2$ are bounded classical solutions of $$-\Delta u = \lambda u + f \quad {\rm in} \; \Omega,$$
where $f\in C(\Omega)$, satisfying $u_1 \le u_2$ on $\partial \Omega$, then $u_1 \le u_2$ in $\Omega$. In particular for $f=0$, considering $u_2=0$, we conclude that if $u_1\le 0$
on $\partial \Omega$, then $u_1 \le 0$ in $\Omega$. This is a special case of maximum principle studied, for instance, by Berestycki, Nirenberg and Varadhan \cite{BNV} in bounded domains of $\mathbb{R}^n$ for subsolutions of a large class of second order elliptic equations $Lu = -\lambda u$. They proved this maximum principle provided $\lambda < \lambda_1(L,\Omega)$, where $\lambda_1(L,\Omega)$ is a sort of eigenvalue defined for nondivergent operators.

The comparison principle also implies the uniqueness of bounded solutions to the Dirichlet problem  
$$
 \left\{ \begin{array}{rcl}  -\Delta u & = & \lambda u  + f\quad {\rm in } \quad \Omega \\[5pt]
                                  u & = & \varphi \quad {\rm on } \quad \partial \Omega  \end{array} \right.
$$
for $\lambda \in [0,\lambda_1]$.
Since any $L^p(\Omega)$ $\lambda$-harmonic function in Hadamard manifolds is bounded for $p \ge 2$ (see \cite{BKT}), this uniqueness result also holds in $L^p(\Omega)$ for $p \ge 2$.

Besides Theorem \ref{Theo_NoBoundedSolutionForNullHausdorff} can be used to prove uniqueness of Green's functions of the operator $-\Delta - \lambda $ that are bounded outside a ball containing the singularity. Indeed suppose that $G_x^1$ and $G_x^2$ are Green's functions of a domain $\Omega$ with $H^{(n-1)/2}(\partial_{\infty} \Omega)=0$ associated to the point $x \in \Omega$. If they are bounded outside a ball centered at $x$, then $G_x^1-G_x^2$ is a bounded solution of \eqref{definitionOfLambdaHarmonic}.
Therefore, it must be zero, proving that $G_x^1=G_x^2$. This result does not hold for a general domain, because if a domain has a bounded $\lambda$-harmonic function $u$ and  such a $G_x^1$, then $G_x^1 + c\, u$, $c \in \mathbb{R}$, is a family of Green's functions with the same property. The behavior of Green's function for points far way the singularity is studied in \cite{AA1} for Hadamard manifolds. 

\

Theorem \ref{Theo_NoBoundedSolutionForNullHausdorff} is optimal in the sense that if the asymptotic boundary 
 has dimension larger than $(n-1)/2$, then Problem \ref{definitionOfLambdaHarmonic} may have a bounded solution.
In fact, for any $s \in (\frac{n-1}{2}, n-1)$ we prove the existence of open sets $\Omega \subset \mathbb{H}^n$ such that the dimension of $\partial_{\infty} \Omega$ is $s$ and for which there exist bounded solutions that decay exponentially at infinity. (For $n=5$ and $s=3> (5-1)/2$, we present in Section \ref{sec-exis} a $\lambda$-harmonic function of a hiperannuli, that is a domain which possess an asymptotic boundary of dimension $3.$) More precisely, we prove in Theorem \ref{theo-exist} that any truly $s-$dimensional subset (see Definition \ref{def-sconj}) of $\partial_\infty\mathbb{H}^n$ is the asymptotic boundary of a domain that admits a bounded $\lambda-$harmonic function.
These two results give a good relation between existence of solution and the dimension of the asymptotic boundary. 

\

In this work, some preliminaries are presented in Section \ref{sec-prelim}. They include the concept asymptotic boundary of $\mathbb{H}^n$, the cone topology and the Hausdorff dimension of a subset of the asymptotic boundary.
Theorem \ref{Theo_NoBoundedSolutionForNullHausdorff} is proved in Section \ref{sec-nonex}. In Section \ref{sec-exis} we present conditions on a subset $X \subset \partial_{\infty} \mathbb{H}^n$ that guarantee the existence of an open set $\Omega$ that admits a bounded $\lambda$-harmonic function
 such that the asymptotic boundary of $\Omega$ is $X$. Then we build examples in which the asymptotic boundary has dimension $s$ for any $s \in \left( \frac{n-1}{2}, n-1 \right)$. For the case $s=n-1$ simple examples are exposed in Section \ref{sec-exis}.

\section{Prelimaries about $\mathbb{H}^n$}\label{sec-prelim}

\subsection{The asymptotic boundary of $\Hi^n$}

To define the asymptotic boundary of $\Omega,$ $\partial_\infty \Omega,$ we follow the ideas of Eberlein and O'Neill \cite{EO} and consider $\partial_\infty \Hi^n$ as the set of equivalence classes of geodesic rays, where we say that $\gamma_1\sim \gamma_2,$ iff $d(\gamma_1(t), \gamma_2(t))$ is bounded, where $t$ is the arclength. Then the closure of $\Hi^n$ is $\overline{\mathbb{H}^n}=\Hi^n \cup \partial_\infty \Hi^n$ and the cone topology is introduced by saying that any open set of 
$\Hi^n$ is open in $\overline{\Hi^n}$ and truncated cones $C(o, \gamma_0, \theta,R)$ are also open. 
For an arclength parametrized geodesic ray $\gamma_0,$ with $\gamma_0(0)=o,$ the truncated cone $C(o, \gamma_0, \theta,R)$ of opening $\theta$ centered at $\gamma_0$ is the union of the two following sets

$$\begin{array}{l}
C(o, \gamma_0, \theta,R)\cap \Hi^n=\{\gamma(t)\,|\, \gamma \text{ is a geodesic, } \gamma(0)=o, |\gamma'(0)|=1, \sphericalangle(\gamma'(0), \gamma_0'(0))<\theta/2 \text{ and } t>R\}
\end{array}$$

and   $$C(o, \gamma_0, \theta,R)\cap \partial_\infty \Hi^n=\{[\gamma]\,|\, \gamma(t)\in C(o, \gamma_0, \theta,R)\cap \Hi^n\text{ for large }t\}.$$
The asymptotic boundary of a set $\Omega\subset \mathbb{H}^n,$ $\partial_\infty \Omega,$ is the boundary with respect to the cone topology of $\Omega$ in $\mathbb{H}^n\cup \partial_\infty \mathbb{H}^n$ minus its usual boundary.

With this notion of asymptotic boundary, we may define the Hausdorff dimension of a subset $X\subset \partial_\infty \Hi^n.$
We start with sets $X$ of measure zero.
For that, remind that for a given point $o \in \mathbb{H}^n$, we may identify $\partial_{\infty}\mathbb{H}^n$ with $S_1(o)=\{p\in \mathbb{H}^n | d(p,o)=1 \}.$
A point $[\gamma]$ in $\partial_{\infty} \mathbb{H}^n$ is identified with  $f_o([\gamma])\in S_{1}(o)$ intercepting the geodesic $\gamma$ in $[\gamma]$ that starts in $o$ with $S_1(o).$ Then $\partial_{\infty} \mathbb{H}^n$ can be identified with $S_{1}(o)$ by the function $f_o:\partial_{\infty} \mathbb{H}^n \to S_{1}(o)$. Observe that if $f_1:\partial_{\infty} \mathbb{H}^n \to S_1(o_1) $ and $f_2:\partial_{\infty} \mathbb{H}^n \to S_1(o_2) $ are the functions associated to the points $o_1,o_2 \in \mathbb{H}^n$ respectively, then $f_2 \circ f_1^{-1}$ is a bijection between  $S_1(o_1)$ and $S_1(o_2)$. Indeed, $f_2 \circ f_1^{-1}$ is a Lipschitz function according to Proposition 1.3 of \cite{SY}.

As a consequence, suppose that $X \subset \partial_{\infty}\mathbb{H}^n$ is such that $f_1(X)\subset S_1(o_1)$ is a set of $r$-dimensional Hausdorff measure zero. Then the $r$-dimensional Hausdorff measure of $f_2(X)\subset S_1(o_2)$ is also zero. Besides, if $Y \subset \partial_{\infty}\mathbb{H}^n$ is such that $f_1(X)\subset S_1(o_1)$ is a set of $r$-dimensional Hausdorff measure finite, then the $r$-dimensional Hausdorff measure of $f_2(X)\subset S_1(o_2)$ is also finite. Therefore the definitions below are well-posed.

\begin{definition} 
The $r$-dimensional Hausdorff measure of $X \subset \partial_{\infty}\mathbb{H}^n$ is zero, denoted by $H^r(X)=0$, if the $r$-dimensional Hausdorff measure of $f_{1}(X) \subset S_1(o_1)$ is zero for some (and therefore for all) identifications of $X.$
\end{definition}

\begin{definition} 
The $r$-dimensional Hausdorff measure of $X \subset \partial_{\infty}\mathbb{H}^n$ is finite, denoted by $H^r(X)<\infty$, if the $r$-dimensional Hausdorff measure of $f_{1}(X) \subset S_1(o_1)$ is finite for some (and therefore for all) identifications of $X.$
\end{definition}

Moreover, we may define the Hausdorff dimension of a subset of $\partial_\infty\mathbb{H}^n$ as the Hausdorff dimension of $f_1(X)$ for an identification of $X.$ This is well defined because the Hausdorff dimension of $X$ is $s(X)=\inf\{r\geq 0\,\vert\, H^r(X)=0\}.$

Some $r$ dimensional subsets of $\partial_\infty\mathbb{H}^n$ are truly $r-$dimensional, and for them we obtain a sharp result in the sense that they are the asymptotic boundary of a domain that admits a bounded $\lambda_1-$harmonic function iff $r>(n-1)/2.$ Their  precise definition is 

\begin{definition}\label{def-sconj}
For $(n-1)/2 < s < n-1,$ a set $X$ contained in $\partial_\infty \mathbb{H}^n$ is a truly $s-$dimensional set if:
\begin{description}
\item[i)] The Hausdorff dimension of $X$ is $s$;
\item[ii)] $H^{s}(X)$ is finite;
\end{description}
and $X$ admits an identification $f:\partial_\infty \mathbb{H}^n\rightarrow S_1(o)$ such that
\begin{description}
\item[iii)] For any $p_0 \in f(X),$
$$ \liminf_{r \to 0} \frac{H^{s}(f(X) \cap B_r(p_0))}{r^{s}} > 0, $$
for $B_r(p_0)$ the ball centered at $p_0$ in $T_o\mathbb{H}^n$ with radius $r;$  
\item[iv)] There exists $K=K(f(X)) >0$ such that for any $p_0 \in f(X)$ and $r>0,$
$$ H^{s}(f(X) \cap B_r(p_0))\le K{r^{s}}. $$
\end{description}
\end{definition}

A totally geodesic hypersphere in $\Hi^n$ is a $(n-1)-$dimensional totally geodesic submanifold of $\Hi^n.$ A totally geodesic hyperball is a region bounded by a totally geodesic hypersphere. Given a totally geodesic hyperball $\mathcal{I}$ bounded by $\partial\mathcal{I},$ the distance with sign to $\partial\mathcal{I}$ is defined as the hyperbolic distance to $\partial\mathcal{I}$ with positive sign in $\mathcal{I}$ and negative sign in $\mathbb{H}^n\backslash \mathcal{I}.$ The asymptotic boundary of a hypersphere is homeomorphic to $\mathbb{S}^{n-2}$ and therefore has Hausdorff dimension $n-2.$

A horosphere in $\Hi^n$ is a $(n-1)-$dimensional submanifold of $\Hi^n$ obtained as the limit set of a sequence of geodesic spheres centered along a geodeisic ray $\gamma(t)$ that contain $\gamma(0).$ It may be seen as a sphere centered at $[\gamma]\in \partial_\infty \mathbb{H}^n.$ A horoball $\mathcal{H}$ is the region bounded by a horosphere that contains the geodesic ray $\gamma((0,\infty)),$ 
The distance with sign to $\partial\mathcal{H}$ is defined as the hyperbolic distance to $\partial\mathcal{H}$ with positive sign in $\mathcal{H}$ and negative sign in $\mathbb{H}^n\backslash \mathcal{H}.$ The asymptotic boundary of a horosphere is only one point so that it has Hausdorff dimension $0.$

\subsection{The Poincar\'e ball model}

The Poincar\'e ball model of $\mathbb{H}^n$ consists in endow the ball $B=B_1(0)\subset\mathbb{R}^n$ of radius one centered at the origin with a Riemannian metric given by 

$$g_{ij}(p)=\frac{4}{(1-|p|^2)^2}\delta_{ij},$$ where $| \cdot |$ is the euclidean distance to the origin. The advantage of this model is that the Hausdorff measure in $\partial_\infty \mathbb{H}^n$ can be seen in $\partial B.$ It becomes then natural to integrate in subsets of $\partial_\infty\mathbb{H}^n$ by integrating in subsets of $\partial B,$ as we will proceed in Section \ref{sec-exis}.

In this model the geodesics are part of circles that cross $\partial B$ orthogonally and diameters though the origin. The totally geodesic hyperspheres are $(n-1)-$spheres that also intercept $\partial B$ orthogonally and the horospheres are spheres contained in $\overline{B}$ that touch $\partial B$ only at one point, called the center of the horosphere. Given $\alpha \in [-1,1),$ we denote by $H_{z,\alpha}$ the horoball centered at $z\in \partial B$ with $\alpha z\in \partial H_{z,\alpha}.$ We remark that given $z$ and $\alpha,$ there is a unique horoball with the properties described above.

With the model set we are able to prove two lemmas about the distance to horo- and hyperspheres, where we use that the hyperbolic distance from $p\in B$ to the origin 0 of $B$ is

\begin{equation}\label{eq-d_aozero}
d(p,0)=\ln\left(\frac{1+|p|}{1-|p|}\right).
\end{equation}

\begin{lemma}\label{lem-disthiprnomodelo}
Let $0\in B$ be the center of $B$ in the ball model. Given $z\in \partial B$ and a positive $\theta,$ let $C(z, \theta, 0)$ be the cone with vertex at $0$, opening angle $\theta$ with axis the ray that connects $0$ to $z.$ Let $I(z,\theta, 0)$ be the totally geodesic hyperball such that 
$\overline{I(z, \theta, 0)}\cap \partial B = \overline{C(z,\theta, 0)}\cap \partial B.$

There exist constants $C_1,$ $C_2$ and $\theta_0,$ such that the hyperbolic distance $d$ satisfies $$ \ln \left( \frac{C_1}{\theta} \right) \le d(0, \partial I(z,\theta,0)) \le \ln \left( \frac{C_2}{\theta} \right), $$
for $\theta < \theta_0.$
\end{lemma}

\begin{proof}
In this setting, $\partial I(z,\theta, 0)$ is, if $\theta<\pi,$ part of an euclidean sphere in $\mathbb{R}^n$ of radius $\tan(\theta/2)$ and center outside $B,$ on the line though $0$ and $z,$ at a distance $\sqrt{\tan^2(\theta/2)+1}=\sec(\theta/2)$ from $0.$ Therefore the euclidean distance from the origin $0$ to $\partial I(z,\theta, 0)$ is $\sec(\theta/2)-\tan(\theta/2).$
From expression \ref{eq-d_aozero}, $$d(0, I(z,\theta,0))=\ln\left(\frac{1+\sec(\theta/2)-\tan(\theta/2)}{1-\sec(\theta/2)+\tan(\theta/2)}\right),$$
which behaves as stated above for $\theta$ close to zero.
\end{proof}

\begin{lemma}\label{lem-disthornomodelo}
The hyperbolic distance (with sign) between $x\in B$ and a horoball $H_{z,0}$ though the origin is
$$d(x,\partial H_{z,0})= \ln \left(  \frac{1-|x|^2}{|z-x|^2} \right). $$
\end{lemma}

\begin{proof}
We can suppose that $z=e_1=(1,0, \dots, 0) \in \partial B$ and use the representation $x = (x_1, y)$, where $y=(x_2, \dots, x_n)$.
Observe that there exists only one $\alpha$ such that the horosphere $\partial H_{z,\alpha}$ contains $x$.  This $\alpha$ corresponds to the intersection between $\partial H_{z,\alpha}$ and the $x_1$-axis, that is given by $(\alpha, 0, \dots, 0)= \alpha \,e_1$, where
\begin{equation}
 \alpha = \frac{x_1 -|x|^2}{1-x_1}. 
 \label{alphaExpression}
\end{equation}
Since the two horospheres $\partial H_{z,\alpha}$ and $\partial H_{z,0}$ are equidistant, the distance between $x$ and $\partial H_{z,0}$ is the same as the distance between $\alpha \, e_1$ and $0$, that is given by
$$ d(p,\partial H_{z,0})= d(\alpha e_1, 0) = \ln \left(\frac{1 + \alpha}{1-\alpha} \right). $$
From \eqref{alphaExpression}, we have
$$ \frac{1 + \alpha}{1-\alpha} =\frac{1-|x|^2}{1-2x_1+|x|^2}= \frac{1-|x|^2}{|z-x|^2}.$$Therefore,
$$ d(x,\partial H_{z,0})= \ln \left(  \frac{1-|x|^2}{|z-x|^2} \right). $$
\end{proof}

\section{Non-existence results}\label{sec-nonex}

In this section we prove the non existence of bounded $\lambda-$harmonic functions in domains of $\Hi^n$ that are unbounded and have a sufficiently small asymptotic boundary. We first prove a weaker result for the case of the asymptotic boundary being a single point and then we generalize it for small sets. For this, we need the $\lambda-$harmonic functions associated to hyperballs.

Given a totally geodesic hyperball $\mathcal{I}\subset \Hi^n,$ if $u:\mathbb{H}^n\rightarrow \R$ is a function that depends only on the distance (with sign) $d$ to the hypersphere $\partial\mathcal{I},$ then the Laplacian of $u$ is given by
$$\Delta u(p)=u''(d(p))+(n-1)\tanh(d(p))u'(d(p)).$$

Hence a $\lambda-$harmonic function in $\mathbb{H}^n$ that depends only on $d$ satisfies
\begin{equation}\label{eq_edo_variavel_d}
h'' (d) + (n-1) \tanh(d) h'(d) = -\lambda h(d). 
\end{equation}

\begin{lemma}
Given $\lambda\in [0,\lambda_1],$ there are constants $d_0>0$ and $C_3>0$ such that the ODE \eqref{eq_edo_variavel_d}
admits a solution $h:(0,\infty)\rightarrow\R$ which is positive and decreasing in $[d_0, +\infty)$ and satisfies $h(d) < C_3e^{-rd}$ for all $d>d_0.$ The exponent $r$ is $r=\frac{n-1}{2}+\frac{n-1}{2}\sqrt{1-\lambda/\lambda_1}.$
\label{lem-perfilAutofuncaoHiperbola}
\end{lemma}

\begin{proof}
Consider the change of variables $t=\dfrac{1}{\sinh(d)},$ which brings the behaviour of a solution at infinity to the origin. With this change, \eqref{eq_edo_variavel_d} becomes

\begin{equation}\label{eq_edo_variavel_t}
(1+t^2)v''(t)+\left(\frac{2t^2+2-n}{t}\right)v'(t)+\left(\frac{\lambda}{t^2}\right)v(t)=0,
\end{equation}
solvable by the Fr\"obenius method.

The two (linearly independent, if $\lambda<\lambda_1)$ solutions to \eqref{eq_edo_variavel_t} are

$$v_1(t)=\sum_{i=0}^\infty a_it^{i+r_1} \text{   and   } 
v_2(t)=\sum_{i=0}^\infty b_it^{i+r_2}, $$

with $a_0\neq 0,\; r_1=r$ and $r_2=r-(n-1)\sqrt{1-\lambda/\lambda_1}.$

Observe that $v_1(0)=0$ and if we choose $a_0>0,$ there is a positive $t_0$ such that $v_1$ is an increasing function in $(0,t_0).$ Besides, decreasing $t_0$ if necessary, we may find a constant $D>0,$ such that
$v_1(t)\leq Dt^{r_1}$ in $(0,t_0).$

Hence, taking $h(d)=v_1\left({1}/{\sinh(d)}\right),$ $h$ is a solution to \eqref{eq_edo_variavel_d} such that
$h>0$ in $\left(d_0,+\infty\right)$ where $d_0=\arcsinh \left(1/t_0\right)$ and 
$$h(d)=v_1\left(\dfrac{1}{\sinh(d)}\right)\leq D\left(\dfrac{1}{\sinh(d)}\right)^{r_1}\leq C_3e^{-r_1d}$$ in this interval.

\end{proof}

\begin{definition} For a given totally geodesic hyperball $\mathcal{I}$  in $\mathbb{H}^n,$ define a $\lambda$-harmonic function $w_\mathcal{I}$ in $\mathbb{H}^n$ by
$$ w_\mathcal{I}(x) = \frac{h(d(x))}{h(d_0)}. $$ 
\label{definicaoAutofuncaoHiperbola}
\end{definition}

The function $w_\mathcal{I}$ is a $\lambda-$harmonic function in $\mathbb{H}^n$ that attains value 1 on the hypersphere $d_0$ apart from $\partial\mathcal{I}$ and decreases exponentially to zero for $d>d_0.$

\begin{theorem}\label{Theo_NoBoundedSolutionforOnePointAtTheBoundary}
Let $\Omega$ be a domain of $\mathbb{H}^n$ such that $\partial_{\infty} \Omega = \{p\}$. Then there is no nontrivial bounded $\lambda$-harmonic function that vanishes on $\partial \Omega$ for any $\lambda \in [0,\lambda_1].$
\end{theorem}

\begin{proof}
Assume, for the sake of contradiction, that $\Omega\subset\Hi^n$ is a domain with $\partial_\infty\Omega=\{p\}$ that admits a bounded $\lambda-$harmonic function $u.$ 
Without loss of generality, we may suppose $\sup \; u = 1/2.$ Let $o$ be some point of $\Omega$ such that $u(o)>0.$

Let $h$ be the solution to \eqref{eq_edo_variavel_d} from Lemma \ref{lem-perfilAutofuncaoHiperbola}. Since $h \to 0$ as $d \to +\infty$, there is some $d_1>d_0$ with $$\frac{h(d_1)}{h(d_0)} < u(o). $$

Let $\mathcal{S}$ be the totally geodesic hypersphere centered at $p$ with
$ d(o,\mathcal{S}) = d_1.$ Now consider $\mathcal{I}$ the totally geodesic hyperball bounded by $\mathcal{S}$ that does not contain $p$ in its asymptotic boundary. Let $w_{\mathcal{I}}$ be the $\lambda$-harmonic function of $\mathcal{I}$ from Definition \ref{definicaoAutofuncaoHiperbola}. We find a compact subset $\Omega_1\subset\Omega \cap \mathcal{I}$ in which $u-w_\mathcal{I}$ is a positive bounded $\lambda-$harmonic function.

Take
$$ O=\{ x \in \Omega\cap \mathcal{I} \; : \: d(x, \mathcal{S}) > d_0 \} .$$

Observe that $O$ is a bounded non empty set, since $\Omega \cap \mathcal{I}$ is bounded and $o\in O.$ Besides 
$$ w_{\mathcal{I}}(o) = \frac{h(d_1)}{h(d_0)} < u(o).$$

Furthermore, $w_\mathcal{I} > u$ on $\partial O$, since $w_{\mathcal{I}} > 0 = u$ on $\partial O \cap \partial \Omega$ and 
$w_{\mathcal{I}} = 1 > u $ on $\partial O \cap \Omega.$  Therefore, there exists some domain $\Omega_1 \subset O$ such that
$$ u - w_{\mathcal{I}} > 0 \quad {\rm in} \quad \Omega_1$$
and $u - w_{\mathcal{I}} = 0$ on $\partial \Omega_1$. Hence $v=u - w_{\mathcal{I}}$ is a positive $\lambda-$harmonic function of the bounded domain $\Omega_1,$ contradicting the fact that $\lambda\leq \lambda_1\left(\Hi^n\right).$

\end{proof}

In order to prove the stronger version of this theorem, we need to estimate $w_{\mathcal{I}}$ at a point by the size of $\partial_\infty\mathcal{I}.$

\begin{lemma}\label{lem-horofunclimitada}
Let $I=I(x,\theta,0)$ be a totally geodesic hyperball in the ball model $B.$ Then the $\lambda-$harmonic function $w_{I}:B\rightarrow \R$ associated to $I$ satisfies
$$ 0 < w_I(0) \le C_4 \theta^{(n-1)/2}, $$
for any $\theta < \theta_1,$ 
 for some constants $C_4 > 0$ and $\theta_1>0.$ 
\end{lemma}

\begin{proof}
According to Lemma \ref{lem-perfilAutofuncaoHiperbola}, 
$$h(d) \le C_3 e^{-(n-1)d/2 }$$
for $d > d_0$ and, from Lemma \ref{lem-disthiprnomodelo}, 
$$d=d(0, \partial I) > \ln \frac{C_1}{\theta} $$
for $\theta < \theta_0$. Choose $\theta_1 < \theta_0$ such that $\ln\left(\frac{C_1}{\theta_1}\right) > d_0$. Hence, for $\theta < \theta_1$, both inequalities are satisfied.
Therefore, 
$$ w_I(0)=h(d(0)) \le C_3 e^{-(n-1)d(0)/2} \le C_3 e^{-(n-1)/2 \ln (C_1/\theta)} = C_4 \theta^{(n-1)/2}. $$
\end{proof}

\begin{theorem}  
Let $\Omega \subset \mathbb{H}^n$ such that $H^{(n-1)/2}(\partial_{\infty} \Omega)=0$. Then there is no nontrivial bounded $\lambda$-harmonic function that
vanishes on $\partial \Omega$ for $\lambda \in [0,\lambda_1]$. 
\end{theorem}

\begin{proof}
We follow the same idea as in Theorem \ref{Theo_NoBoundedSolutionforOnePointAtTheBoundary}. Assume that there exists such a function $u,$ that we suppose satisfies $\sup \; u = 1/2.$
Let $B$ be the ball model of $\mathbb{H}^n.$ Without loss of generality, we assume that for $0$ the origin of $B,$ $0 < u(0) \le 1/2.$

Let $X'\subset \partial B$ be the asymptotic boundary of $\Omega$ in this model. Since the $(n-1)/2$-dimensional Hausdorff measure of $X'$ is zero, for any $\varepsilon > 0$ there exists a countable collection of balls in $\partial B,$ $U_{\theta_{i}} (x_i) =C(x_i,\theta_i,0) \cap \partial B$ that covers $X'$, such that $\theta_i < \theta_1$ and 

\begin{equation} \sum_{i}    \theta_i^{(n-1)/2} < \varepsilon. \label{thetaisumcontrol}
\end{equation} 
Consider the hyperballs $I_i=I(x_i,\theta,0)$ and the associated $\lambda_1-$harmonic function $w_i= w_{I_i}$. Then, from Lemma \ref{lem-horofunclimitada}, 
$$ 0 \le w_i(0) \le  C_4 \theta_i^{(n-1)/2}. $$
Using \eqref{thetaisumcontrol}, we get
$$ 0 < \sum_{i} w_i(0) \le C_4 \varepsilon. $$ 
Taking $\varepsilon$ small, it follows that 
$$ 0 < \sum_{i} w_i(0) < u(0). $$
Let $O = V_{d_0} \cap \Omega $, where $V_{d_0} = \{ q \in \mathbb{H}^n \; : \; d(q, I_i) > d_0 \; \forall i \}$. 
Observe that $w_i > 0$ in $V_{d_0}$ and $\sum_{i=1}^N w_i(p)$ is bounded in compacts of $V_{d_0}$ uniformly in $N,$ by Harnack Inequality. Hence $w=\sum_i w_i$ is well defined, positive and, from the regularity theory, $\lambda_1$-harmonic function. 

Defining $v=u-w$ in $O$ we have that $v(o) > 0$ and $v < 0$ on $\partial O$ since $w > 0=u$ on $\partial O \backslash \Omega$ and $w \ge 1 > u$ on $\partial O \cap \Omega$. Moreover $O$ is bounded. As before, we obtain a contradiction. 
\end{proof}

\section{Existence results}\label{sec-exis}

The purpose of this section is to show that there exist subsets of $\partial_\infty\mathbb{H}^n$ of dimension $s\in (\frac{n-1}{2},n-1],$ that are the asymptotic boundary of domains that admit bounded $\lambda_1-$harmonic functions. We are especially interested in the case $s\in (\frac{n-1}{2},n-1)$. For that, given a truly subset $X$ of $\partial_\infty \mathbb{H}^n$ of dimension $s,$ we construct a $\lambda_1-$harmonic function $u:\mathbb{H}^n\to \mathbb{R}$ positive in $\Omega$ with $\partial_\infty \Omega=\overline{X},$ as stated in the next theorem.

\begin{theorem}\label{theo-exist}
Let $X$ be a truly $s-$dimensional subset of $\partial_\infty\mathbb{H}^n$ for some $(n-1)/2 < s < n-1.$ Then there is a bounded $\lambda_1$-harmonic function $u$ in $\mathbb{H}^n,$ positive in some set $\Omega$ such that $\partial_{\infty}\Omega = \overline{X}.$ 
Moreover, for any $p\in \Omega$ $$ |u(p)| \le M_0(d+1) \, e^{-\left(s - \frac{(n-1)}{2}\right)d}, $$
where $d=d(p,o)$ for some $o\in \mathbb{H}^n$ fixed and $M_0$ is a positive constant that depends only on $K(f(X)),$ $n$ and $s.$
\end{theorem}

For proving this result we integrate $\lambda_1$-harmonic functions that are constant along horospheres. To find them, let $\mathcal{H} \subset \mathbb{H}^n$ be a horoball and $d:\mathbb{H}^n \to \mathbb{R}$ be the distance with sign (positive in $\mathcal{H}$) to $\partial\mathcal{H}.$ 
Then a $\lambda_1-$harmonic function that is constant on horospheres equidistant to $\partial\mathcal{H}$ has the form $u_{\mathcal{H}}(p)=u_{\mathcal{H}}(d(p))$ and its Laplacian is 
$$\Delta u_{\mathcal{H}}(p) = u_{\mathcal{H}}''(d(p)) + (n-1) u_{\mathcal{H}}'(d(p)). $$
Therefore, $h=u_{\mathcal{H}}\circ d$ satisfies
$$h''(d) + (n-1) h'(d) = - \lambda_1 h(d). $$
Hence,
\begin{equation}
 u_{\mathcal{H}}(p) = u_{\mathcal{H}}(d(p)) = (A_1 + A_2 d(p) \, ) e^{(n-1)d(p)/2},
\label{solutionInHoroballs}
\end{equation}
for constants $A_1$ and $A_2.$

The next lemma, a consequence of Lemma \ref{lem-disthornomodelo}, presents an expression of $u_{\mathcal{H}}$ in the Poincar\'e ball model.

\begin{lemma}
In the Poincar\'e ball model, any $\lambda_1-$harmonic function given by \eqref{solutionInHoroballs} associated to the horoball $H_{z,0}$ can be expressed  by
\begin{equation}
u_{z,A_1, A_2}(x) = \left[ \frac{1-|x|^2}{|z - x|^2} \right]^{(n-1)/2} \left( A_1 + A_2 \ln\left[ \frac{1-|x|^2}{|z-x|^2} \right] \right).
\label{uzA1A2}
\end{equation}
\end{lemma}

The result follows from expression \eqref{solutionInHoroballs} and Lemma \ref{lem-disthornomodelo}. We denote by $u_z$ the function above for $A_1=0$ and $A_2=1.$

\begin{proof}[Proof of Theorem \ref{theo-exist}]

We divide the proof into three claims. The first one asserts the absolute integrability of $u_z$, the second provides a way to obtain the required solution $u$, and the third one treats the decay of $u$.

\begin{claim}\label{lem-iv_implica_limitada}
Let $\Lambda\subset \partial B,$ obtained as an identification of $X\subset \partial_\infty \mathbb{H}^n,$ a truly $s-$dimensional set not necessarily satisfying condition iii) from Definition \ref{def-sconj}. There is $M > 0$ depending only on $n$ and $s,$ such that

\begin{equation}
 \int_\Lambda \left| u_z(x)\right| \, dH^{s}(z) < M K(\Lambda) ,
\label{boundForHausdorffIntegral}
\end{equation}
holds for any $x \in B.$

\end{claim}

For $x \in B$, let $\delta=1-|x|$ and $k \in \mathbb{N}$ such that $2 < 2^k \delta \le 4$. Define
$$  \Lambda_i=\{ z \in  \Lambda \, : \, 2^{i-1}\delta \le |z-x| < 2^{i}\delta \} \quad {\rm for } \quad i \in \{ 1, 2, \dots, k\}.$$ 
Observe that
$$ \Lambda = \bigcup_{i=1}^k \Lambda_i, $$
since $\delta \le |z-x| \le 2$ for any $z \in \partial B$. If $z \in \Lambda_i$, we have $|z-x| \ge 2^{i-1}\delta$ and, therefore
\begin{equation}
\frac{1-|x|^2}{|z-x|^2} = \frac{(1 +|x|)(1-|x|)}{|z-x|^2} \le \frac{2\delta}{2^{2(i-1)}\delta^2} = \frac{8}{4^{i}\delta}.
\label{part0OfIntegrandEstimate}
\end{equation}
Hence
\begin{equation*}
\left[\frac{1-|x|^2}{|z-x|^2} \right]^{\frac{n-1}{2}} \le \left(\frac{8}{4^{i}\delta}\right)^{\frac{n-1}{2}}.
\end{equation*}
Moreover,  
$$ \frac{1-|x|^2}{|z-x|^2} = \frac{(1 +|x|)(1-|x|)}{|z-x|^2} \ge \frac{\delta}{4}$$
since $|z-x| \le 2$. Using this inequality and that for $a\le \min\{1,t\}$
$$|\ln t| \le \ln t - 2\ln a,$$  we obtain 
$$
\left| \ln \left[ \frac{1-|x|^2}{|z-x|^2} \right] \right| \le \ln \left[ \frac{1-|x|^2}{|z-x|^2} \right] - 2 \ln \left(\frac{\delta}{4}\right).
$$
Then, the monotonicity of $\ln x$ and \eqref{part0OfIntegrandEstimate} imply that
\begin{equation}
\left| \ln \left[ \frac{1-|x|^2}{|z-x|^2} \right] \right| \le \ln \left( \frac{8}{4^{i}\delta} \right) - 2 \ln \left(\frac{\delta}{4}\right) = \ln \left( \frac{128}{4^i \delta^3} \right) \le \ln \left( \frac{128}{\delta^3} \right).
\label{part2OfIntegrandEstimate}
\end{equation}
From \eqref{part0OfIntegrandEstimate} and \eqref{part2OfIntegrandEstimate}, we conclude that for $z\in \Lambda_i$
\begin{equation}
 |u_z(x)|  =  \left[\frac{1-|x|^2}{|z-x|^2} \right]^{\frac{n-1}{2}} \left| \ln \left[ \frac{1-|x|^2}{|z-x|^2} \right] \right| \le \left(\frac{8}{4^{i}\delta}\right)^{\frac{n-1}{2}} \ln \left( \frac{128}{\delta^3} \right).
\label{ModulusIntegrandEstimate}
\end{equation}
Now we estimate $H^s(\Lambda_i)$. Let $x_0 = \frac{x}{|x|} \in \partial B$. Notice that if $z \in \Lambda_i,$
$$|z-x_0| \le |z-x|+|x-x_0| =|z-x|+\delta \le 2|z-x|
\le 2^{i+1}\delta. $$
Then $\Lambda_i \subset B_{2^{i+1}\delta}(x_0) \cap \Lambda$ and from condition {\it (iv)} of Definition \ref{def-sconj},
\begin{equation}
H^s(\Lambda_i) \le K(\Lambda) (2^{i+1}\delta)^s.
\label{HsCiEstimate}
\end{equation}
Using \eqref{ModulusIntegrandEstimate} and \eqref{HsCiEstimate}, we conclude
\begin{equation}
 \int_{\Lambda_i} |u_z(x)|\, dH^s(z)  \le \left(\frac{8}{2^{2i}\delta}\right)^{\frac{n-1}{2}} \ln \left( \frac{128}{\delta^3} \right)  K(\Lambda) (2^{i+1}\delta)^s=K_0 2^{i(s-(n-1))} \delta^{s-\frac{n-1}{2}} \ln \left( \frac{128}{\delta^3} \right),
 \label{eqIntegralEstimateForModulusOfuzOnCi}
\end{equation}
where $K_0= 8^{\frac{n-1}{2}}2^s K(\Lambda)$. Therefore,
$$ \int_{\Lambda} |u_z(x)|\, dH^s(z) = \sum_{i=1}^k \int_{\Lambda_i} |u_z(x)|\, dH^s(z) \le K_0 \delta^{s-\frac{n-1}{2}} \ln \left( \frac{128}{\delta^3} \right) \sum_{i=1}^k 2^{i(s-(n-1))}.$$
Since $s-(n-1) < 0$, we get $\sum_{i=1}^k 2^{i(s-(n-1))} <  \frac{2^{(s-(n-1))}}{1-2^{(s-(n-1))}}$. Thus
\begin{equation}
 \int_{\Lambda} |u_z(x)|\, dH^s(z) \le K_1 \delta^{s-\frac{n-1}{2}} \ln \left( \frac{128}{\delta^3} \right),
\label{IntegralEstimateForModulusOfuz}
\end{equation}
where $K_1=K_0 \frac{2^{(s-(n-1))}}{1-2^{(s-(n-1))}} $. Using that the right-hand side is uniformly bounded for $\delta \in (0,1)$, we conclude the claim.

The second claim states that the integral in $\Lambda$ of the family of functions 
$u_{z},$ $z\in \Lambda,$ given by $\eqref{uzA1A2}$ defines a bounded $\lambda_1$-harmonic function.

\begin{claim}\label{lem-limitacao}
If $\Lambda$ is a truly $s-$dimensional set and there is $M > 0$ such that \eqref{boundForHausdorffIntegral} holds for any $x \in B$, then the function
\begin{equation}
 u(x) = \int_\Lambda u_z (x) \; dH^{s}(z) ,
\label{uDefinitionByIntegral}
\end{equation}
is a bounded $\lambda_1$-harmonic function, positive in some set $\Omega$ such that $\partial_{\infty}\Omega =\partial \Omega \cap \partial B= \overline{\Lambda}$. 
\end{claim}

From \eqref{boundForHausdorffIntegral}, $u$ is well defined and bounded by $M K(\Lambda)$. Moreover, since $H^{s}(\Lambda)$ is finite 
and $u_z(x)$ is $C^{\infty}$ in $z$ and $x$, we have that $u$ is $C^{\infty}$ and it is a $\lambda_1$-harmonic function. 

Let $\Omega =\{ x \in B \; : \; u(x) >0 \}$. First we prove that $\partial_{\infty}\Omega \subset \overline{\Lambda}$. For $z_0 \in \partial B \backslash \overline{\Lambda}$
and a positive $r < \frac{1}{2} d (z_0, \overline{\Lambda})$ that will be chosen later, let $V$ be the ball centered at $z_0$ with radius $r$.
Hence, 
$$| z - x | > \frac{1}{2}d(z_0, \overline{\Lambda}) \quad {\rm for \; any } \; \; x \in V \cap B \; \; {\rm and }\;\; z \in \Lambda. $$
Naming $\bar{d}=d(z_0,\overline{\Lambda})$ and using that $|x| > 1-r$ for $x \in V \cap B$,  we have
$$ \ln\left[ \frac{1-|x|^2}{|z-x|^2} \right] < \ln (1 - (1-r)^2) - \ln \left( \frac{\bar{d}^2}{4}\right) = \ln \left( \frac{8r - r^2}{\bar{d}^2} \right) ,$$
for $z \in \Lambda$. If $r < \bar{d}^2/8$, the right-hand side of this inequality is negative and, therefore, 
$$u_z(x) < 0 \quad {\rm for}  \; \; x \in V \cap B \; \;  {\rm and} \; \; z\in \Lambda.$$
Then $u(x) < 0$ for $x \in V \cap B$. As a consequence $z_0 \not\in \partial_{\infty}\Omega$ which proves that $\partial_{\infty}\Omega \subset \overline{\Lambda}$.

\

We prove now that $\partial_{\infty}\Omega \supset \overline{\Lambda}$. For $x_0 \in \overline{\Lambda}$ let $x=(1-\delta) x_0$, where $ 0 < \delta < 1/8$ will be chosen later.
Observe that $\delta = 1-|x|$. Consider the sets $\Lambda_i$ defined in Claim 1 for $i\in \{1, \dots, k\}$, where $2 < 2^k \delta \le 4$.
\\ \\
\noindent If $z \in \Lambda_i$ for $i < (k-2)/2$, we have that $|z-x|< 2^i \delta$ and the definition of $k$ implies that $2^{2i} \delta < 1$. Then
\begin{equation}
\ln \left[ \frac{1-|x|^2}{|z-x|^2} \right] \ge \ln \left[ \frac{(1-|x|)(1+|x|)}{|z-x|^2} \right] \ge \ln \left( \frac{\delta }{2^{2i}\delta^2} \right)=\ln \left( \frac{1}{2^{2i}\delta} \right) >0.
\label{eqlnOnCiByBelow}
\end{equation}
Hence $u_z(x) > 0$ and therefore
\begin{equation}
\int_{\Lambda_i} u_z(x) \; dH^s(z) \ge 0 \quad {\rm for} \quad i < \frac{k-2}{2}.
\label{eqonCiNEW} 
\end{equation}
We can improve this estimative in $\Lambda_1$. Indeed, if $z \in \Lambda_1$, then $|z-x| < 2 \delta$. Using this and \eqref{eqlnOnCiByBelow},
\begin{equation}
 u_z(x)= \left[\frac{1-|x|^2}{|z-x|^2} \right]^{\frac{n-1}{2}}  \ln \left[ \frac{1-|x|^2}{|z-x|^2} \right] \ge \left[\frac{1}{4\delta} \right]^{\frac{n-1}{2}}  \ln \left( \frac{1}{4\delta} \right)
\label{eqEstForUzOnC1}
\end{equation} 
Observe now that for $z \in \Lambda \cap B_{\delta}(x_0)$, we get $|z - x_0| < \delta$ and, therefore $|z-x| \le |z-x_0|+|x_0 -x | < 2 \delta$.
On the other hand $|z-x| \ge \delta$. Hence 
$$\Lambda \cap B_{\delta}(x_0) \subset \Lambda_1.$$
Condition {\it (iii)} of Definition \ref{def-sconj} implies that there exist $\delta_1 >0$ and $D >0$ such that
$$ H^{s}(\Lambda_1) \ge H^{s}(\Lambda \cap B_{\delta}(x_0)) > D \delta^{s},$$
for $\delta \le \delta_1$. From this and \eqref{eqEstForUzOnC1} we conclude, for $\delta \le \delta_1$, that 
\begin{equation}
\int_{\Lambda_1} u_z(x) \; dH^s(z) \ge \left[\frac{1}{4\delta} \right]^{\frac{n-1}{2}}  \ln \left( \frac{1}{4\delta} \right) D \delta^{s} = \frac{D}{4^{n-1}} \delta^{s -\frac{n-1}{2}} \ln \left(\frac{1}{4\delta}\right) .
\label{eqEstForIntOnC1} 
\end{equation} 
Now we consider the integral on $\Lambda_i$ for $i \ge (k-2)/2$. For that, remind \eqref{eqIntegralEstimateForModulusOfuzOnCi} from which we obtain
$$ \int_{\Lambda_i} u_z(x) \; dH^s(z) \ge - \int_{\Lambda_i} |u_z(x)|\, dH^s(z)  \ge - K_0 2^{i(s-(n-1))} \delta^{s-\frac{n-1}{2}} \ln \left( \frac{128}{\delta^3} \right). $$
Hence, denoting $E = \displaystyle \bigcup_{\frac{k-2}{2} \le i \le k}\Lambda_i$ and using that $s -(n-1) <0$, we have
\begin{align*} \int_{E} u_z(x) \; dH^s(z) &\ge - K_0  \delta^{s-\frac{n-1}{2}} \ln \left( \frac{128}{\delta^3} \right) \sum_{\frac{k-2}{2} \le i \le k } 2^{i(s-(n-1))} \\[5pt]   &\ge - K_0  \delta^{s-\frac{n-1}{2}} \ln \left( \frac{128}{\delta^3} \right) \frac{2^{\frac{k-2}{2}(s-(n-1))}}{1-2^{(s-(n-1))}} = -K_2 \delta^{s-\frac{n-1}{2}} \ln \left( \frac{128}{\delta^3} \right) 2^{\frac{k}{2}(s-(n-1))},
\end{align*}
where $K_2= K_0 \frac{2^{(n-1)-s}}{1-2^{(s-(n-1))}}$. Since $2^k > \frac{2}{\delta}$,
\begin{equation}
 \int_{E} u_z(x) \; dH^s(z) \ge - K_2  \delta^{s-\frac{n-1}{2}} \left( \frac{\delta}{2} \right)^{\frac{(n-1)-s}{2}} \ln \left( \frac{128}{\delta^3} \right).
\label{eqEstForIntOnCiUnionOrE}
\end{equation}
From \eqref{eqonCiNEW}, \eqref{eqEstForIntOnC1} and \eqref{eqEstForIntOnCiUnionOrE}, we get
\begin{equation}
 \int_\Lambda u_z(x) d H^s(z) \ge  \frac{D}{4^{n-1}} \delta^{s -\frac{n-1}{2}} \ln \left(\frac{1}{4\delta}\right) - K_2  \delta^{s-\frac{n-1}{2}} \left( \frac{\delta}{2} \right)^{\frac{(n-1)-s}{2}} \ln \left( \frac{128}{\delta^3} \right) ,
\label{eqonCCC} 
\end{equation}
for $\delta \le \min\{ 1/8, \delta_1 \}$. The right-hand side is positive for $\delta$ close to zero, since the negative term has the extra factor $(\delta/2)^{\frac{(n-1)-s}{2}}$ that converges to zero. That is, there exists
$\delta_2 >0$ such that $u((1-\delta)x_0) > 0$ for $\delta \le \delta_2$. Then $x_0 \in \partial_{\infty} \Omega$, proving the result.

The third and last claim required to prove Theorem \ref{theo-exist} is about the boundedness of the function.

\begin{claim}\label{lem-limitada}
If $\Lambda$ is a truly $s-$dimension set, then the $\lambda_1$-harmonic function $u$ given by \eqref{uDefinitionByIntegral} is positive in some $\Omega$ such that
$\partial_{\infty} \Omega=\overline{\Lambda}$ and 
$$ |u(x)| \le M_0 (d+1) \, e^{-\left(s - \frac{(n-1)}{2}\right)d}, $$
where $d=d(x,0)$ and $M_0$ is a positive constant that depends only on $K(\Lambda)$, $n$ and $s.$
\end{claim}

Inequality \eqref{IntegralEstimateForModulusOfuz} implies that
$$|u(x)| \le K_1 \delta^{s-\frac{n-1}{2}} \ln \left( \frac{128}{\delta^3} \right),$$
for $\delta=1-|x|.$ From \eqref{eq-d_aozero}, we have
$ \delta= 2/(1+ e^d)$ and then $e^{-d} \le \delta \le 2 e^{-d}.$ Therefore
$$|u(x)|  \le K_1 (2e^{-d}) ^{s-\frac{n-1}{2}} \ln \left( \frac{128}{(e^{-d})^3} \right), $$
proving the result.

\end{proof}

\begin{remark}
For $x_0 \in \overline{\Lambda}$ and $x=(1-\delta)x_0$, inequality \eqref{eqonCCC} implies
$$ u(x) \ge C \delta^{s -\frac{n-1}{2}} \ln \left(\frac{1}{\delta}\right), $$
for small $\delta$, where $C$ is some positive constant. Taking $\delta=\frac{2}{1+ e^d}$ as we did before, we conclude that 
$$u(x) \ge \tilde{M}  \,d \, e^{-\left(s - \frac{(n-1)}{2}\right)d}, $$
where the positive constant $\tilde{M}$ depends on $K(\Lambda)$ (given by condition (iii) from definition \ref{def-sconj}), $x_0$, and the set $\Lambda$. From this and the previous claim, we get
$$   \tilde{M}  \,d \, e^{-\left(s - \frac{(n-1)}{2}\right)d} \le u(x) \le M_0 (d+1) \, e^{-\left(s - \frac{(n-1)}{2}\right)d},$$
for $x=(1-\delta)\, x_0$ as $\delta$ is close to zero.
\end{remark}

For $\Omega\subset \mathbb{H}^n$ such that $\partial_\infty \Omega$ has $(n-1)$ Hausdorff dimension, we consider specific subsets. For instance, if 
$\Omega=\mathbb{H}^n$, we can find a radially symmetric $\lambda$-harmonic function $u$ in $\Omega$ by solving an hypergeometric equation. Such a solution is bounded and has exponential decay at infinity (see \cite{GO}) for $\lambda \in (0,\lambda_1]$.  
Moreover $u$ does not change sign, otherwise $\lambda$ would be the first eigenvalue of some ball. Another example can be obtained by taking $u=u_1 - u_2$, where $u_1$ and $u_2$ are radially symmetric $\lambda$-harmonic functions with different points of symmetry, $o_1$ and $o_2$. If $u_1(o_1) = u(o_2) > 0$, then $u$ is positive in some geodesic hyperball $\mathcal{I}$ and negative outside $\mathcal{I}$.

In some special situations, we can find $\lambda_1-$harmonic with a more explicit representation. For instance, consider a $\lambda_1$-harmonic function $u$ presented in Section \ref{sec-nonex} that depends only on $d$, the distance (with sign) to some hypersphere $\partial \mathcal{I}$. If $n=5$, then $u(d)$ satisfies  \eqref{eq_edo_variavel_d} with $\lambda=\lambda_1=\frac{(5-1)^2}{4}=4$. As an example of such a solution, we can take
$$ u(d) = \frac{\cosh d - d \sinh d}{\cosh^3 d}$$
that is an even function with two zeros $d_1$ and $-d_1$, positive in $(-d_1,d_1)$ and negative elsewhere. Then $u$ is a positive $\lambda_1$-harmonic function in $\Omega$, that is the hyperannulus bounded by the two hyperspheres that have distance $d_1$ to $\partial \mathcal{I}$. The asymptotic boundary of $\Omega$ is homeomorphic to $\mathbb{S}^3$ that has dimension $3$.

\subsection{Cantor like subsets of $\partial_\infty \mathbb{H}^n$}

Given  $(n-1)/2<s<n-1,$ $s$ is the Hausdorff dimension of $\partial_\infty\Omega$ for a domain $\Omega$ that admits a bounded $\lambda_1$-harmonic function.
We exhibit truly $s-$dimensional sets in $\partial_\infty\mathbb{H}^n$ and then the result follows from Theorem \ref{theo-exist}.

Let us construct Cantor like sets in $\mathbb{S}^1\subset \partial B.$ Given integers $1\leq l<m,$ the Cantor like set $K_{l,m}$ is constructed as follows:

We introduce coordinates in $\mathbb{S}^1$ and think of it as the interval $[0,1]$ with 0 and 1 identified. 

Let $K_0=[0,1].$
Let $$K_1=\left[0,\frac{1}{m}\right]\cup \left[\frac{1}{l},\frac{1}{l}+\frac{1}{m}\right] \cup \left[\frac{2}{l},\frac{2}{l}+\frac{1}{m}\right] \cup ... \cup \left[\frac{l-1}{l},\frac{l-1}{l}+\frac{1}{m}\right]$$
be a union of $l$ intervals of length $1/m$ equally distributed in $\mathbb{S}^1.$
Then $K_n$ is obtained inductively from $K_{n-1},$ a union of $l^{n-1}$ intervals  of length $1/m^{n-1},$ by taking from each interval of $K_{n-1},$ $l$ intervals of length $1/m^n,$ equally distributed in the interval. 

The generalized Cantor set $K_{l,m}$ is then $$K_{l,m}=\bigcap_{n=1}^{\infty} K_n.$$

It is well known that $K_{l,m}$ has Hausdorff dimension $s$ for $s$ such that $\frac{l}{m^s}=1,$ hence $s=\frac{\ln l}{\ln m}.$ Moreover, the $s-$dimensional Hausdorff measure of $K_{l,m}$ is

$$H^s\left(K_{l,m}\right)=\omega_s \lim_{n\rightarrow\infty}\sum_{i=1}^{l^n}\left(\frac{1}{m^n}\right)^s=\omega_s<\infty,$$
if we think of $K_{l,m}$ as a subset of $[0,1].$ (The constant $\omega_s$ is used in the definition of the Hausdorff measure (see \cite{Fe})
and represents the volume of the unit $s$-dimensional ball when $s$ is a positive integer.)
Since we may identify $\mathbb{S}^1\subset \partial B$ with $[0,1],$ $H^s(K_{l,m})$ is finite as a subset of $\partial_\infty\mathbb{H}^n.$
Besides, if instead of taking intervals of length $1/m$ of the size of the intervals in the step before, one takes intervals of length $a/m$ for $a\in (0,1),$ then the Hausdorff dimension of the Cantor set would be $$s=\frac{\ln l}{\ln m-\ln a}$$
and the $s-$dimensional measure would also be $\omega_s<\infty.$ 
Therefore, for any $s\in (0, 1),$ there is a Cantor like set of Hausdorff dimension $s$ in $\mathbb{S}^1.$

\begin{proposition}
Let $\Lambda$ denote $K_{l,m}$ of dimension $s=\frac{\ln l}{\ln m}.$
For any $z \in \Lambda,$
$$\frac{\omega_s r^s}{l^s} \le H^{s}(\Lambda \cap B_r(z)) \le w_s(2m)^s r^{s}$$
and therefore $\Lambda\subset \partial B$ is a truly $s-$dimensional subset of $\partial B.$
\end{proposition}

\begin{proof}
Given $z \in \Lambda,$ and $r>0,$ take $n\in \mathbb{N}$ such that 
$$\frac{1}{l^{n+1}}\leq r<\frac{1}{l^{n}}.$$

Observe that $B_r(z)$ contains, for some $i\in \mathbb{N},$ $0<i<n-1,$ $I_n=[i/l^n,(i+1)/l^n]$ and therefore $\Lambda\cap B_r(z)$ contains a contracted copy of $\Lambda,$ $K_{n,i}=i/l^n+1/l^n\times \Lambda$ which is part of $K_n.$ Hence
$$H^s(\Lambda\cap B_r(z))\geq H^s(\Lambda\cap I_n)=(1/l^n)^sH^s(\Lambda)\geq \frac{\omega_s r^s}{l^s},$$
concluding the proof of the first inequality.

The second one is trivial if $r> 1/m.$
If $r<1/m,$ take $n\in \mathbb{N}$ such that
\begin{equation}
\frac{1}{m^{n+1}}\leq r<\frac{1}{m^{n}}.
\end{equation}
Notice that $\Lambda$ is contained in $K_n$ and $B_r(z)$ intercepts $K_n$ at most in two consecutive subintervals of $K_{n}$ of length $1/l^n.$ Therefore, 
$$H^s(\Lambda\cap B_r(z)) \leq H^s(K_n\cap B_r(z))\leq \frac{2H^s(\Lambda)}{l^n}\leq H^s(\Lambda)(2rm)^s.$$
\end{proof}

For any $0<s<1,$ we exhibited a truly $s-$dimensional subset $\partial B$ that is, from Theorem \ref{theo-exist}, the boundary of a domain that admits a bounded $\lambda_1-$harmonic function if $s>(n-1)/2.$
The construction of the asymptotic boundary of a domain where a bounded eigenfunction exists ends by observing that if $k$ is an integer such $(n-1)/2<k+s<n-1,$ then a Cantor like set $\Lambda$ of Hausdorff dimension $k+s$ can be constructed as follows:

We may think of $\mathbb{S}^{n-1}$ as a subset of $\mathbb{R}^{n},$
$\mathbb{S}^{n-1}=\{(x_1,...,x_{n})|x_1^2+x_2^2+...+x_{n}^2=1\}$ and it contains 
$\mathbb{S}^1=\{(0,...0,x_{n-1} ,x_{n})|x_{n-1}^2+x_{n}^2=1\}.$ Let $K_{l,m}$ be the Cantor set contained in $\mathbb{S}^1$ and then for $0<\varepsilon<1$ let $$
\begin{array}{l}
X=\{(x_1,...x_k,0,...,0,x_{n-1}, x_{n})|
\frac{1}{1-(x_1^2+x_2^2+...+x_k^2)}(x_{n-1},x_{n})\in K_{l,m}\\
\text{ and }(x_1^2+x_2^2+...+x_k^2)\in [-\varepsilon,\varepsilon]\}.
\end{array}
$$

It is clear that $X$ is diffeomorphic to $K_{l,m}\times [-\varepsilon,\varepsilon]^k$ which has Hausdorff dimension $s+k$ since any  covering $\{A_t\}_{t\in I}$ of $K_{l,m}$ induces a covering of $X$ by taking $\{A_t\times [-\varepsilon,\varepsilon]^k\}_{t\in I}.$ Besides, $X$ is also a truly $(k+s)-$dimensional set and therefore, it is the asymptotic boundary of a domain $\Omega$ that admits a bounded positive $\lambda_1-$harmonic function. This can be summarized in the following result

\begin{theorem}
Given  $s\in\left((n-1)/2,n-1\right),$ $s$ is the Hausdorff dimension of $\partial_\infty\Omega$ for a domain $\Omega$ that admits a bounded $\lambda_1$-harmonic function.
\end{theorem}

\end{document}